\documentclass[12pt,reqno]{article}

\usepackage[usenames]{color}
\usepackage{amssymb}
\usepackage{graphicx}
\usepackage{amscd}

\usepackage[colorlinks=true,
linkcolor=webgreen,
filecolor=webbrown,
citecolor=webgreen]{hyperref}

\definecolor{webgreen}{rgb}{0,.5,0}
\definecolor{webbrown}{rgb}{.6,0,0}

\usepackage{color}
\usepackage{fullpage}
\usepackage{float}

\usepackage{graphics,amsmath,amssymb}
\usepackage{amsthm}
\usepackage{amsfonts}
\usepackage{latexsym}
\usepackage{epsf}

\newcommand{\seqnum}[1]{\href{http://oeis.org/#1}{\underline{#1}}}

\begin{document}

\begin{center}
\vskip 1cm{\LARGE\bf Fixed Sequences for a Generalization of the \\
\vskip .07in
Binomial Interpolated Operator and for some \\
\vskip .13in
Other Operators}
\vskip 1cm
\large
Marco Abrate, Stefano Barbero, Umberto Cerruti, and Nadir Murru\\
Department of Mathematics \\
University of Turin \\
via Carlo Alberto 8/10 \\
Turin \\
Italy \\
\href{mailto:marco.abrate@unito.it}{\tt marco.abrate@unito.it}\\
\href{mailto:stefano.barbero@unito.it}{\tt stefano.barbero@unito.it}\\
\href{mailto:umberto.cerruti@unito.it}{\tt umberto.cerruti@unito.it}\\
\href{mailto:nadir.murru@unito.it}{\tt nadir.murru@unito.it}\\
\end{center}

\theoremstyle{plain}
\newtheorem{theorem}{Theorem}
\newtheorem{corollary}[theorem]{Corollary}
\newtheorem{lemma}[theorem]{Lemma}
\newtheorem{proposition}[theorem]{Proposition}

\theoremstyle{definition}
\newtheorem{definition}[theorem]{Definition}
\newtheorem{example}[theorem]{Example}
\newtheorem{conjecture}[theorem]{Conjecture}

\theoremstyle{remark}
\newtheorem{remark}[theorem]{Remark}

\begin{abstract}  
This paper is devoted to the study of eigen-sequences for some
important operators acting on sequences. Using functional equations
involving generating functions, we completely solve the problem of
characterizing the fixed sequences for the Generalized Binomial operator.
We give some applications  to integer sequences. In particular we show
how we can generate fixed sequences for Generalized Binomial and their
relation with the Worpitzky transform. We illustrate this fact with
some interesting examples and identities, related to Fibonacci,
Catalan, Motzkin and Euler numbers. Finally we find the
eigen-sequences for the mutual compositions of the operators
Interpolated Invert, Generalized Binomial and Revert.
\end{abstract}

\section{Introduction}
In a previous paper, Barbero, Cerruti, and Murru \cite{bcm2} have studied the action of a generalization of the Binomial interpolated operator on linear recurrent sequences. Particular attention was given to the study of the fixed sequences for this operator, finding some interesting results about known sequences, like the Catalan numbers. As final remarks, we posed the question of studying the general problem of characterizing the sequences fixed by this operator. 
The research interest involving operators which act on sequences has an important aspect related to those sequences which are fixed, also called \emph{eigen-sequences}. Bernstein and Sloane \cite{BerSlo} have deeply studied fixed sequences, for the composition of several operators as Binomial, Invert and others. The Binomial inverse transformation has been especially studied in this sense. For example, Sun \cite{Sun} and Wang \cite{Wang} have found interesting results about the fixed sequences for the Binomial inverse operator. Furthermore, Mattarei and Tauraso \cite{Mat} have studied several congruences using these invariant sequences. Since the Binomial and the Binomial inverse transform arise as a particular case of Generalized Binomial operator, the relevance of finding the corresponding eigen-sequences is clear. Here we give a complete answer for this question characterizing the generating functions of such fixed sequences.
\begin{definition}
We define, over an integral domain $R$, the set $$\mathcal{S}(R)=\left\{a=(a_n)_{n=0}^{+\infty}: \ \forall n  \ a_n \in R ,\, a_0 =1\right\}\ .$$
\end{definition}
\begin{definition} \label{Idef}
For any element $x \in R$, the \emph{Interpolated Invert} operator \ $I^{(x)}$ acts on a sequence $a \in \mathcal{S}(R)$ 
giving the sequence $I^{(x)}(a) = b =\left(b_n(x)\right)_{n=0}^{ +\infty }$, where the sequence $b$ has generating function
\begin{equation} \label{invertgen}
\mathbf{B}(t) = \sum\limits_{n = 0}^{ + \infty } {b_n(x) t^n  = \frac{{\sum\limits_{n = 0}^{ + \infty } {a_n t^n } }}{{1 - xt\sum\limits_{n = 0}^{ + \infty } {a_n t^n } }}}\quad.  
\end{equation} 
\end{definition}
\begin{definition}\label{Ldef}
For any element $y \in R$, the \emph{Interpolated Binomial} operator $L^{(y)}$ acts on a sequence $a \in \mathcal{S}(R)$
giving the sequence $L^{(y)}(a)=l=\left(l_n(y)\right)_{n=0}^{ +\infty }$, where the $n$--th element of the sequence $l$ is
\begin{equation} \label{bininterp}
l_n=\sum^{n}_{j=0}{\binom{n}{j} }y^{n - j} a_j.
\end{equation}
The exponential generating function for the sequence $l$ is
$$
\mathcal{L}(t) = \sum\limits_{n = 0}^{ + \infty }l_n\frac{t^n }{n!} = \sum\limits_{n = 0}^{ + \infty } \sum\limits_{j = 0}^n \frac{(yt)^{n - j} }{(n - j)!}\frac{a_j t^j }{j!} =e^{yt} \mathcal{A}(t) 
$$
being
$$ e^{yt}=\sum\limits_{n=0}^{+ \infty}\frac{(yt)^n}{n!} \quad \quad \mathcal{A}(t)=\sum\limits_{n=0}^{+\infty}
\frac{a_n t^n}{n!} , $$
so that (recalling that $a_0=1$) we have  the ordinary generating function
$$
\mathbf{L}(t) =\frac{1}{t}A\left( {\frac{t}{{1 - yt}}} \right)
$$
with $A(t)=\sum\limits_{n=0}^{+\infty}a_nt^{n+1}\quad.$
\end{definition}
\begin{definition}\label{etadef}
The \emph{Revert} operator $\eta$ acts as follows: if we consider the sequences $a=(a_n)_{n=0}^{+\infty}$ $\in \mathcal{S}(R)$ ,  $b =(b_n)_{n=0}^{+\infty}\in\mathcal{S}(R)$ we have  $\eta(a)=b$ if and only if the sequences $a$ and $b$ satisfy the relations 
\begin{equation} \label{e0}
\begin{cases}
u=u(t)=\sum\limits_{n=0}^{+\infty}a_nt^{n+1}\\
t=t(u)=\sum\limits_{n=0}^{+\infty}b_nu^{n+1} \ &\text{inverse series of $u$ .}\\
\end{cases}
\end{equation}
\end{definition}
Prodinger \cite{Prod} defined a generalization of the Interpolated Binomial operator as
\begin{equation} \label{new-L} L^{(h,y)}(a)=b,\quad b_n=\sum_{i=0}^n\binom{n}{i}h^iy^{n-i}a_i\  \end{equation}
for any sequence $a\in\mathcal{S}(R)$ and nonzero elements $h,y\in R$. The definition can be extended when $h=0$ or $y=0$ as follows:
$$L^{(0,y)}(a)=(y^na_0)_{n=0}^{+\infty},\quad L^{(h,0)}(a)=(h^na_n)_{n=0}^{+\infty},\quad L^{(0,0)}(a)=(a_0,0,0,0,...).$$
We observe that the sequence $L^{(h,y)}(a)$ has  exponential generating function  
$$\mathcal{A}_h(t)=e^{yt}\mathcal{A}(ht) $$
and ordinary generating function 
\begin{equation} \label {ogflh} \mathbf{L}_h(t)=\frac{1}{ht}A\left( {\frac{ht}{{1 - yt}}} \right)\ ,\end{equation}
as one can find with a little bit of calculation, where the functions $\mathcal{A}(t)$ and $A(t)$ are the same introduced in the previous Definition \ref{Ldef}.
The authors \cite{bcm2} have characterized linear recurrent sequences of degree 2 fixed by this operator and this result could be extended to linear recurrent sequences of higher degree. However, as we pointed out in that article, there are fixed points for the operator $L^{(h,y)}$ which are not linear recurrent sequences. In this paper we give a full explanation of this fact, via functional equations (see for example Aczel \cite{Aczel} and Kannappan \cite{Plkan}), and we find the fixed sequences related to the operators $I^{(x)}\circ L^{(h,y)}$, $\eta \circ L^{(y,h)}$ and $\eta \circ I^{(x)}$. 

\section{Fixed sequences for Generalized Binomial operator}
We observe that a sequence $a \in \mathcal{S}(R)\ ,$ fixed by the operator $L^{(h,y)}$, has an exponential generating function $\mathcal{A}(t)$ which must satisfy the functional equation
\begin{equation}\label{funclh} \mathcal{A}(t)=\mathcal{A}_h(t)=e^{yt}\mathcal{A}(ht)\end{equation}
with the condition $\mathcal{A}(0)=a_0=1\ .$
The analysis of this equation allows us to state some results and conditions on parameters $h$ and $y$, summarized in the following theorem.
\begin{theorem} \label{main-theorem}
When $h\neq1$, $h\neq-1$, and $y\neq0$ the only sequence fixed by the operator $L^{(h,y)}$ is
\begin{equation} A=\left\{\left(\frac{y}{1-h}\right)^n\right\}_{n=0}^{+\infty}.\end{equation}
When $h=1$ we trivially have fixed sequences if and only if $y=0$, in other words the operator $L^{(h,y)}$ becomes the identical operator.
When $y=0$, $h\neq1$, and $h\neq-1$ we have the only fixed sequence $A=\left\{a_0=1,a_n=0\quad\forall n\geq1\right\}$.
When $y=0$ and $h=-1$ we have fixed sequences $A$ for any even exponential generating function $\mathcal{A}(t)$ with $\mathcal{A}(0)=1$.
Finally, when $y\neq0$ and $h=-1$ the exponential generating function of a fixed sequence fulfills the relation
$$ \mathcal{A}(t)=\phi\left(e^{yt}\right)\ ,$$
where the function $\phi(u)$   satisfies the functional equation 
\begin{equation} \label{phif}\phi(u)=u\phi\left(\frac{1}{u}\right)\quad \mbox{with} \quad u>0\quad \mbox{and}\quad \phi(1)=1\ .\end{equation}
\end{theorem}
\begin{proof}
First of all we consider the case $h\neq1$ , $h\neq -1$, and $y\neq0$. If we substitute $e^{yt}=u$, or equivalently $t=\frac{\log(u)}{y}$, we have
\begin{equation}\label{phibef}\mathcal{A}\left(\frac{\log(u)}{y}\right)=u\mathcal{A}\left(\frac{1}{y}\log(u^h)\right)\ . \end{equation}
Now we naturally define $\phi=\mathcal{A}\circ g \circ \log$, where $g(t)=\frac{t}{y}$, and equation (\ref{funclh}) becomes
\begin{equation}\label{phieq}
\phi(u)=u\phi(u^h)\ ,
\end{equation}
where the condition $\mathcal{A}(0)=1$ corresponds to the condition $\phi(1)=1\ .$
Iterating the previous relation (\ref{phieq}), we easily obtain
$$\phi(u)=u\phi(u^h)=u^{1+h}\phi(u^{h^2})=u^{1+h+h^2}\phi(u^{h^3})=\cdots=u^{\frac{1-h^n}{1-h}}\phi(u^{h^n})\ .$$
Hence, if $|h|<1$, taking the limit when $n\rightarrow +\infty$, we find 
$$\phi(u)=u^{\frac{1}{1-h}}\phi(1)=u^{\frac{1}{1-h}}\ ,$$
which implies that
$$\mathcal{A}(t)=e^{\frac{yt}{1-h}}=\sum_{n=0}^{+\infty}{\left(\frac{y}{1-h}\right)^n\frac{t^n}{n!}}\ .$$
The case $|h|>1$ gives in a similar way the same result. 
When $h=-1$ and $y\neq0$ the functional equation (\ref{phif}) is straightforward from equation (\ref{phieq}).
If $h=1$ the equation (\ref{funclh}) becomes 
$$\mathcal{A}(t)=e^{yt}\mathcal{A}(t)\ ,$$
which is obviously satisfied by sequences in $\mathcal{S}(R)$ if and only if $y=0$.
Moreover, when $y=0$ and $h=-1$,  clearly $\mathcal{A}(t)=\mathcal{A}(-t)$, so the exponential generating function $\mathcal{A}(t)$ is an even function.
Finally, when $y=0$, $h\neq1$, and $h\neq-1$  the functional equation $$\mathcal{A}(t)=\mathcal{A}(ht)\ ,$$
with $\mathcal{A}(0)=1$, has the only constant solution $\mathcal{A}(t)=1$, corresponding to the fixed sequence $A=\left\{a_0=1, a_n=0\quad\forall n\geq1\right\}$.
\end{proof}
With the next theorem we characterize all the solutions  $\phi \in \mathcal{C}^{\infty}\left(\mathbb{R}^{+}\right)$ of (\ref{phif}) with $\phi(1)=1\ .$ 
\begin{theorem}\label{solphi}
The equation (\ref{phif}) has a solution $\phi \in \mathcal{C}^{\infty}\left(\mathbb{R}^{+}\right)$ with $\phi(1)=1$ if and only if exists an $\mathbb{R}$-valued function $F\in \mathcal{C}^{\infty}\left(\mathbb{R}^{+}\times\mathbb{R}^{+}\right)$ such that
\begin{description}
\item{  i) $F(1,1)=1$}
\item{ ii) $F(u,v)=F(v,u)$}
\item{iii) $F(cu,cv)=cF(u,v)\quad \forall c \in \mathbb{R}^{+}.$} 
\end{description}
Then $\phi(u)=F(u,1)=F(1,u)$.
\end{theorem}
\begin{proof}
Let us consider an $\mathbb{R}$-valued function $F\in\mathcal{C}^{\infty}\left(\mathbb{R}^{+}\times\mathbb{R}^{+}\right)$.
From the symmetry $F(u,v)=F(v,u)$ and using the homogeneity of degree 1 of $F(u,v)$, we have $$uF\left(1,\frac{v}{u}\right)=vF\left(1,\frac{u}{v}\right)$$
and consequently
$$F\left(1,\frac{u}{v}\right)=\frac{u}{v}F\left(1,\frac{v}{u}\right)\ .$$
In particular, when we pose $v=1$ the previous equality becomes
$$F(1,u)=uF\left(1,\frac{1}{u}\right)$$
and clearly $\phi(u)=F(1,u)=F(u,1)$ is a solution of equation (\ref{phif}) belonging to $\mathcal{C}^{\infty}\left(\mathbb{R}^{+}\right)$.
Now if $\phi \in \mathcal{C}^{\infty}\left(\mathbb{R}^{+}\right) $ is a solution of equation (\ref{phif}) we have
\begin{equation}\label{phiuv}\phi\left(\frac{u}{v}\right)=\frac{u}{v}\phi\left(\frac{v}{u}\right)\quad \forall (u,v)\in \mathbb{R}^{+}\times\mathbb{R}^{+}\end{equation}
and we can define an $\mathbb{R}$-valued function $F\in\mathcal{C}^{\infty}\left(\mathbb{R}^{+}\times\mathbb{R}^{+}\right)$
by posing
$F(u,v)=v\phi\left(\frac{u}{v}\right)\ .$
This function satisfies $F(1,1)=1$, is symmetric, in fact from relation (\ref{phiuv})
$$F(u,v)=v\phi\left(\frac{u}{v}\right)=u\phi\left(\frac{v}{u}\right)=F(v,u)\ ,$$
and homogeneous of degree $1$ because
$$F(cu,cv)=cv\phi\left(\frac{cu}{cv}\right)=cv\phi\left(\frac{u}{v}\right)=cF(u,v) \quad \forall c \in \mathbb{R}^{+}\ .$$
\end{proof}
\section{Some remarks, examples, and applications to integer sequences}
All the previous results, with the same conditions on parameters $y$ and $h$, can be easily generalized in order to find fixed sequences whose first element $a_0$ belongs to $\mathbb{R}^{*}$.
We also observe that there are some straightforward methods to generate solutions of equation (\ref{funclh}) starting from known solutions, when $h=-1$ and the parameter $y$ has an assigned value. A simple calculation proves that if the function $\mathcal{A}(t)$ satisfies equation (\ref{funclh})  when $h=-1\ ,$ also the product $f(t)\mathcal{A}(t)$, with an even function $f(t)$, is a solution, and any linear combination of solutions is also a solution for the same assigned value of the parameter $y$.
Finally if we have a set of functions $\mathcal{A}_k(t)$ satisfying $\mathcal{A}_k(t)=e^{y_kt}\mathcal{A}_k(-t)$,  $k=1,2,\cdots,n$, then 
$\mathcal{A}(t)=\prod\limits_{k = 1}^n {\left(\mathcal{A}_k(t)\right)^{\alpha_k}}$, $\alpha_k \in \mathbb{N}$,  satisfies $\mathcal{A}(t)=e^{t\sum\limits_{k=1}^n{\alpha_{k} y_{k}}}\mathcal{A}(-t)\ .$ 
At last we have to consider the problem related to the existence of fixed sequences starting with exactly $k$ zeros, or in other words sequences with $\mathcal{A}^{j}(0)=0$ when $j=0,1,\cdots,k-1$ and $\mathcal{A}^{(k)}(0)\neq 0$. 
Differentiating  $k$-times both members of equation (\ref{funclh}) with respect to the variable $t$, we have 
$$\mathcal{A}^{(k)}(t)=\sum\limits_{j=0}^k{{\binom{k}{j}}y^{j}e^{yt}h^{k-j}\mathcal{A}^{(k-j)}(ht)}$$
and if $\mathcal{A}(t)$ is an exponential generating function of a fixed sequence starting with $k$ zeros, evaluating this equality for $t=0$, we obtain $\mathcal{A}^{(k)}(0)=h^{k}\mathcal{A}^{(k)}(0)$. This relation is true when $k$ is odd, if and only if $h=1$ and consequently $y=0$, the trivial case, and when $k$ is even we retrieve the non trivial case $h=-1$. So we have fixed sequences starting with $k$ zeros for $L^{(-1,y)}$ if and only if $k=2m$. This fact allows us to state that a fixed sequences with $2m$ starting zeros is a  shift of $2m$ position  of a fixed sequence with $a_0\neq0$, whose exponential generating function is a solution of equation (\ref{funclh}) when $h=-1$. This is a particular case of what we have pointed out before, because the exponential generating function of such a sequence is the product of a known solution of equation (\ref{funclh}), when $h=-1$, with the even function $f(t)=t^{2m}$.\\ After these remarks let us illustrate some straightforward application of Theorem \ref{solphi} to integer sequences. If we consider any real number $\alpha$ and we take
$$\phi(u)=\frac{u^{\alpha}+u^{1-\alpha}}{2}$$
the sequence
$$\left(\frac{\alpha^{n}+(1-\alpha)^n}{2}\right)_{n=0}^{+\infty}$$
is fixed for $L^{(-1,1)}$. In particular, when $\alpha=\frac{1+\sqrt{5}}{2}$, the sequence becomes $$\left(\frac{L_n}{2}\right)_{n=0}^{+\infty}\ ,$$ where $L_n$ is the $n$-th Lucas number. Another interesting situation is $\alpha=\frac{1}{2}$, in this case $$\phi(u)=\sqrt{u}$$ corresponding to $$\mathcal{A}(t)=e^{\frac{yt}{2}}\ .$$ Thus all the sequences having exponential generating functions 
\begin{equation}\label{besselegf}\mathcal{A}(t)=e^{pt}I_{\nu}(qt)t^{\pm\nu}\quad p,q\in \mathbb{R}^*,\quad \nu \in\mathbb{Z} \ ,\end{equation} where the function $I_{\nu}(x)$ is the modified Bessel function of the first kind (see Abramowitz and Stegun \cite{Abram}) , are eigen-sequences for the operator $L^{(-1,2p)}$, because the functions $I_{\nu}(x)x^{\pm\nu}$ are even functions. In the OEIS \cite{Sloane} there are a lot of sequences whose exponential generating functions correspond to one of those in equation (\ref{besselegf}) for suitable values of $p \ ,$ $q$, and $\nu$. The Motzkin numbers \seqnum{A001006} and the central binomial coefficients \seqnum{A000984} have exponential generating functions
$$\mathcal{A}(t)=e^tI_1(2t)t^{-1}\quad \text{and}\quad\mathcal{A}(t)=e^{2t}I_0(2t)$$ and they are eigen-sequences respectively for the operators $L^{(-1,2)}$ and $L^{(-1,4)}$. Moreover the sequences of the form $$(P_n(\sqrt{m})(k\sqrt{m})^n)_{n=0}^{+\infty}\ ,$$ (for example the central Delannoy numbers \seqnum{A001850} and \seqnum{A115865}), where the polynomial $P_n(x)$ is the $n$-th Legendre polynomial (see  Abramowitz and Stegun \cite{Abram}) , have exponential generating functions 
$$ \mathcal{A}(t)=e^{kmt}I_0\left(k\sqrt{m(m-1)}t\right)$$ and they are invariant for the operators $L^{(-1,2km)}$.\\
As a final application we want to show how to generate many fixed sequences starting from the operator $L^{(h,y)}$, finding some interesting formulas.
Let us consider an even sequence $$b=(b_0,0,b_2,0,b_4,0,b_6,0,...)$$ we know that the sequence $a=L^{(h,y)}(b)$ has exponential generating function 
$$e^{yt}\sum_{n=0}^{+\infty} \cfrac{b_nh^nt^n}{n!}\ ,$$
i.e., has an exponential generating function given by $e^{yt}$ multiplied by an even function. Thus, as immediate consequence of the Theorem \ref{main-theorem}, the sequence $a$ is a fixed sequence for the operator $L^{(-1,2y)}$. From Definition \ref{new-L} we find a relation between the sequences $a$ and $b$
\begin{equation} \label{ff} \sum_{i=0}^n\sum_{k=0}^{\left\lfloor i/2\right\rfloor}\binom{n}{i}\binom{i}{2k}(-1)^i2^{n-i}y^{n-2k}h^{2k}b_{2k}=a_n,\quad \forall n\geq1. \end{equation}
This equality  gives a connection with the Worpitzky transformation \cite{L}, which converts any sequence $s$ into a polynomial sequence $s(x)$ as follows:
$$W(s)=s(x),\quad s_n(x)=\sum_{k=0}^n\sum_{m=0}^k(-1)^m\binom{k}{m}s_k(x+m+1)^n.$$
We observe that the polynomial sequence $s(x)$ is an Appell sequence, i.e.,
$$L^{(1,y)}(s(x))=s(x+y).$$
When we consider an even sequence $b$
there is an unique sequence $a$, which we call the \emph{worpification} of $b$ such that
$$W(a)=a(x), \quad a(0)=b.$$
Since the polynomial sequences obtained by the Worpitzky transformation are Appel sequences, we have
$$L^{(1,y)}(b)=L^{(1,y)}(a(0))=a(y)$$
and so we know that $a(y)$ is fixed by $L^{(-1,2y)}$.
Let us consider the Catalan numbers, \seqnum{A000108}, $$C=(1, 1, 2, 5, 14, 42, 132, 429, 1430,\ldots)\ .$$ The sequence $b$ of Catalan numbers interpolated with 0's, i.e.,
$$b=(1,0,1,0,2,0,5,0,14,0,42,0,...)$$
and the sequence $M$ of Motzkin numbers \seqnum{A001006} satisfy the relation $L^{(1,1)}(b)=M$ (see OEIS \cite{Sloane}, \seqnum{A000108}, Paul Barry 2003). Thus Motzkin numbers are fixed by $L^{(-1,2)}$, as we have previously showed, and using (\ref{ff}) we have the following relation between Motzkin and Catalan numbers
$$\sum_{i=0}^n\sum_{k=0}^{\left\lfloor i/2\right\rfloor}\binom{n}{i}\binom{i}{2k}(-1)^i2^{n-i}C_k=M_n.$$
In particular, using our techniques we have found that
$$\sum_{k=0}^{\left\lfloor i/2\right\rfloor}\binom{i}{2k}C_k=M_i$$
which is a known result and moreover we have 
$$\sum_{i=0}^n\binom{n}{i}(-1)^i2^{n-i}M_i=M_n.$$
The last relation can be found directly as a consequence of the Corollary 3.3 of the work of Sun \cite{Sun}. Furthermore, we have a connection between Catalan and Motzkin numbers in terms of the Worpitzky transformation. Indeed, if we call $a$ the worpification of $b$ and $W(a)=a(x)$, then we have that
$$a(0)=b, \quad a(1)=M,\quad a(2)=C.$$
Another application involves, for example, the sequence $F$ of Fibonacci numbers \seqnum{A000045}. The Binomial transform of the aerated Fibonacci numbers 
$$b=(0,0,1,0,1,0,2,0,3,0,5,0,8,0,13,0,21,\ldots)\ ,$$
is the sequence $a=L^{(1,1)}(b)$ corresponding to the sequence \seqnum{A101890} (see OEIS \cite{Sloane}, \seqnum{A000108}, Paul Barry 2004). Thus, we know that $a=L^{(-1,2)}(a)$ and we find the formulas
$$a_n=\sum_{i=0}^n\sum_{k=0}^{\left\lfloor i/2\right\rfloor}\binom{n}{i}\binom{i}{2k}(-1)^i2^{n-i}F_k$$
and
$$a_n=\sum_{i=0}^n\binom{n}{i}(-1)^i2^{n-i}a_i.$$
\section{Fixed sequences for Interpolated Invert composed with Generalized Binomial}
The aim of this section is to find all the sequences fixed by the operator $I^{(x)}\circ L^{(h,y)}$. We find a functional equation whose solutions gives us the expression of the ordinary generating functions of these fixed sequences and it determines some conditions which the parameters $x$,$y$ and $h$ must satisfy to have solutions.
First of all, we recall that for any sequence $A \in \mathcal{S}(R)$ the ordinary generating function $\mathbf{L}_h(t)$ for $B=L^{(h,y)}(A)$, satisfies equation (\ref{ogflh}). The ordinary generating function of $C=I^{(x)}(B)$ as we defined in equation (\ref{invertgen}), becomes 
$$ \mathbf{C}(t)=\frac{\mathbf{L}_h(t)}{1-xt\mathbf{L}_h(t)}$$
and we have a fixed sequence if and only if  
\begin{equation}\label{funcil}\mathbf{A}(t)=\frac{\mathbf{L}_h(t)}{1-xt\mathbf{L}_h(t)}\end{equation}
where $\mathbf{A}(t)=\frac{1}{t}A(t)$ is the ordinary generating function of $A$ and $A(t)=\sum\limits_{n=0}^{+\infty}{a_nt^{n+1}}$. Substituting the  expression of $\mathbf{L}_h(t)$ in equation (\ref{funcil}), with some algebraic manipulations, we find the relation
$$A(t)=\frac{A\left(\frac{ht}{1-yt}\right)}{h-xA\left(\frac{ht}{1-yt}\right)}$$ 
which is equivalent to 
\begin{equation}\label{funcil1} A\left(\frac{ht}{1-yt}\right)=\frac{hA(t)}{1+xA(t)}\ .\end{equation}
Now we study the functional equation (\ref{funcil1}) remembering that if $A\in \mathcal{S}(R)$ we have $A(0)=0$ and $A^{'}(0)=1$. We find explicitly the solutions in the following
\begin{theorem}
The operator $I^{(x)}\circ L^{(h,y)}$ fixes the sequences $A\in \mathcal{S}(R)$ whose ordinary generating functions are of the form 
$$ \mathbf{A}(t)=\frac{1}{1+\left(\frac{x+y}{h-1}\right)t}$$
if $x\neq0$ and $h\neq0,1$,  and of the form
$$ \mathbf{A}(t)=\frac{1}{1-ct}$$
if $x\neq0$, $h=1$, $x+y=0$, $c \in \mathbb{R}\ .$ 
\end{theorem}
\begin{proof}
We have to solve the functional equation (\ref{funcil1}),
so differentiating both members of this equation with respect to $t$ we obtain
$$\frac{h}{\left(1-yt\right)^2}A^{'}\left(\frac{ht}{1-yt}\right)=\frac{hA^{'}(t)}{\left(1+xA(t)\right)^2}$$
multiplying both members by $ht^2$ and multiplying and dividing the second member by $\left(A(t)\right)^2 \ ,$ the previous relation becomes
$$\left(\frac{ht}{1-yt}\right)^2A^{'}\left(\frac{ht}{1-yt}\right)=\frac{t^2A^{'}(t)}{\left(A(t)\right)^2}\left(\frac{hA(t)}{1+xA(t)}\right)^2\ .$$
Now if we  use equation (\ref{funcil1}) to rewrite the second member and we pose $w=\frac{ht}{1-yt}$ we find
$$w^2A^{'}(w)=\frac{t^2A^{'}(t)}{\left(A(t)\right)^2}\left(A(w)\right)^2\ , $$
or better
\begin{equation}\label{funcw}\frac{w^2A^{'}(w)}{\left(A(w)\right)^2}=\frac{t^2A^{'}(t)}{\left(A(t)\right)^2}\ .\end{equation}
Finally we come to relation (\ref{funcw}) which implies that 
$$\frac{t^2A^{'}(t)}{\left(A(t)\right)^2}=c_1$$ where $c_1$ is a real constant. Moreover $c_1=1$ because the condition
$A^{'}(0)=1$ gives
$$c_1=\mathop {\lim }\limits_{t \to 0}\left(\frac{t^2A^{'}(t)}{\left(A(t)\right)^2}\right) =\frac{1}{A^{'}(0)}=1\ .$$
Then  we can easily solve the differential equation
$$\frac{t^2A^{'}(t)}{\left(A(t)\right)^2}=1$$ 
through separation of variables, obtaining 
\begin{equation}\label{funcsol}A(t)=\frac{t}{1-ct}\ .\end{equation}
Substituting the solutions (\ref{funcsol}) into equation (\ref{funcil1}) we find
$$\frac{ht}{1-(y+ch)t}=\frac{ht}{1-(c-x)t}$$
and this relation is true if and only if $y+ch=c-x$ or equivalently $c(h-1)=-(x+y) \ .$
Clearly if $h\neq1$ we have $c=-\frac{x+y}{h-1}$, else if $h=1$ this equality is verified for all $c \in \mathbb{R}$ if and only if $x+y=0$.
\end{proof}
\begin{remark}
We observe that the sequences $a$ which have generating functions satisfying the hypotheses of the previous Theorem have the following terms
$$a_n=(-1)^n\left(\cfrac{x+y}{h-1}\right)^n \quad \text{or}\quad a_n=c^n$$
respectively when $x\neq0$ and $h\neq0,1$ or when $x\neq0\ ,$ $h=1\ ,$ $x+y=0\ ,$ $c\in \mathbb{R}\ .$
\end{remark}
\section{Fixed sequences for Revert composed with Generalized Binomial and for Revert composed with Interpolated Invert}
In order to determine the fixed sequences for the operator $\eta \circ L^{(h,y)}$ we need to find the functional equation related to their ordinary generating function. We recall that the $\eta$ operator acts as we have described in equations (\ref{e0}), so for a fixed sequence $A\in \mathcal{S}(R)$ we have
\begin{equation}\label{etalh}
\left\{ {\begin{array}{*{20}c}
   {u = t\mathbf{L}_h (t)}  \\
   {t = u\mathbf{A}(u)}  \\
\end{array}} \right.
\end{equation}
where $\mathbf{L}_h(t)$ is the ordinary generating function of the sequence $L^{(h,y)}(A)$ defined in equation (\ref{ogflh}) and $\mathbf{A}(u)=\frac{1}{u}A(u)$ is the ordinary generating function of $A$. 
Combining the two relations (\ref{etalh}) we finally come to the desired functional equation
\begin{equation}\label{funcetalh} A\left(\frac{1}{h}A\left(\frac{ht}{1-yt}\right)\right)=t \ ,\end{equation}
The study of this equation and the relative conditions on parameters $y$ and $h$ for the resolution are summarized in the following theorem
\begin{theorem}\label{revgenbin}
The operator $\eta \circ L^{(h,y)}$ has fixed sequences if and only if $h=1$ and their ordinary generating functions are
$$ \mathbf{A}(t)=\frac{1}{1+\frac{y}{2}t}\ . $$
\end{theorem}
\begin{proof}
If we pose in equation (\ref{funcetalh}) $w=\frac{ht}{1-yt}$, we obtain the equivalent form
\begin{equation}\label{firstep} A\left(\frac{1}{h}A(w)\right)=\frac{w}{h+yw}\ . \end {equation}
Multiplying both members with $\frac{1}{h}$ and applying $A$ to the two sides of this equation we have
\begin{equation}\label{secondstep} A\left(\frac{1}{h}A\left(\frac{1}{h}A(w)\right)\right)=A\left(\frac{w}{h(h+yw)}\right)\ .
\end{equation}
Now if we use relation (\ref{firstep}) with $\frac{1}{h}A(w)$ instead of $w$ we find
\begin{equation}\label{sost}A\left(\frac{1}{h}A\left(\frac{1}{h}A(w)\right)\right)=\frac{\frac{1}{h}A(w)}{1+\frac{y}{h}A(w)}\ , \end{equation}
and with a little bit of calculation equation (\ref{secondstep}) becomes
\begin{equation}\label{thirdstep} A\left(\frac{w}{h(h+yw)}\right)=\frac{A(w)}{h+yA(w)}\ .\end{equation}
Differentiating this relation with respect to $w$ gives
$$\frac{1}{(h+yw)^2}A^{'}\left(\frac{w}{h(h+yw)}\right)=\frac{hA^{'}(w)}{\left(h+yA(w)\right)^2}$$
and evaluating this equality for $w=0$, remembering that $A(0)=0$ and $A^{'}(0)=1$, we have
$$\frac{1}{h^2}=\frac{1}{h}$$
so there are solutions if and only if $h=1$.
With this information used in equation (\ref{thirdstep}) we retrieve a functional equation 
$$ A\left(\frac{w}{1+yw}\right)=\frac{A(w)}{1+yA(w)}$$
similar to equation (\ref{funcil1}) when $h=1$, which gives the solutions 
$$ A(w)=\frac{w}{1-cw}\ .$$
Between these solutions we have to determine the real constant $c$ in order to satisfy equation (\ref{firstep}) when $h=1$. A simple substitution yelds
$$\frac{w}{1-2cw}=A(A(w))=\frac{w}{1+yw}$$
which is true if and only if $c=-\frac{y}{2}$ giving $A(w)=\frac{w}{1+\frac{y}{2}w}$.
\end{proof}
An analogous result is true when we consider the problem of finding fixed sequences $A \in \mathcal{S}(R)$  for $\eta \circ I^{(x)}$.
The functional equation which arises in this case is 
\begin{equation}\label{funcetai} A\left(\frac{A(t)}{1-xA(t)}\right)=t \end{equation}
and comes out combining the following relations 
$$
\left\{ {\begin{array}{*{20}c}
   u &=& \frac{{t\mathbf{A}(t)}}{{1 - xt\mathbf{A}(t)}}  \\
   t &=& u\mathbf{A}(u)  \\
\end{array}} \right.
$$
where we made use of equations (\ref{e0}), (\ref{invertgen}) and $\mathbf{A}(t)=\frac{1}{t}A(t)$ is the ordinary generating function of $A$.
All the results concerned with the solutions of equation (\ref{funcetai}) and fixed sequences are resumed in the next theorem.
\begin{theorem}\label{revintinv}
The fixed sequences for the operator $\eta \circ I^{(x)}$ have ordinary generating functions 
$$ \mathbf{A}(t)=\frac{1}{1+\frac{x}{2}t}\ . $$
\end{theorem}
\begin{proof}
Setting $w=\frac{A(t)}{1-xA(t)}$ in equation (\ref{funcetai}), clearly $A(t)=\frac{w}{1+xw}$, and we obtain
$$ A(w)=t \ .$$
Applying $A$ to both members of this equation we find
$$ A(A(w))=A(t)=\frac{w}{1+xw} $$
which is of the same kind of equation (\ref{firstep}), that we have already solved , when $h=1$, during the proof of the previous theorem.
\end{proof}
\begin{remark}
We observe that the sequences $a$ which have generating functions satisfying the hypotheses of the previous Theorem \ref{revgenbin} and Theorem \ref{revintinv} respectively have  the following terms
$$a_n=(-1)^n\left( \cfrac{y}{2} \right)^n\quad \text{and} \quad a_n=(-1)^n\left( \cfrac{x}{2} \right)^n\ .$$
\end{remark}

\section{Acknowledgements}
The authors would like to thank the referee for the valuable and interesting comments and suggestions which have improved this paper.

\bigskip
\hrule
\bigskip
\noindent 2000 {\it Mathematics Subject Classification}: Primary 11B37;
Secondary 11B39.
\noindent \emph{Keywords: Binomial operator, Catalan numbers, Fibonacci numbers, Invert operator, Motzkin numbers, recurrent sequences, Revert operator}  

\bigskip
\hrule
\bigskip

\noindent (Concerned with sequences 
\seqnum{A000045},
\seqnum{A000108},
\seqnum{A000984},
\seqnum{A001006},
\seqnum{A001850},
\seqnum{A101890},
\seqnum{A115865}, and
\seqnum{A155585}.)

\bigskip
\hrule
\bigskip

\vspace*{+.1in}
\noindent
Received April 10 2011;
revised version received  August 2 2011.
Published in {\it Journal of Integer Sequences}, September 25 2011.

\bigskip
\hrule
\bigskip

\noindent
Return to
\htmladdnormallink{Journal of Integer Sequences home page}{http://www.cs.uwaterloo.ca/journals/JIS/}.
\vskip .1in

\end{document}